\documentclass[letterpaper, 10 pt, conference]{ieeeconf}  

\IEEEoverridecommandlockouts                              
\usepackage{amsmath}
\usepackage{amssymb}
\usepackage{amsthm}
\usepackage{mathtools}
\usepackage{xifthen}
\usepackage{cite}
\usepackage{algorithm}

\let\labelindent\relax
\usepackage[shortlabels]{enumitem}
\setlist[enumerate]{nosep}

\usepackage{graphicx}
\usepackage{color}

\usepackage{calc}

\usepackage{lipsum}

\newcommand{\tcap}{\text{cap}}
\newcommand{\sumi}{\sum_{i=1}^n}
\newcommand{\sumj}{\sum\nolimits_{j\neq i}^n}

\renewcommand{\vec}[1]{\mathbf{#1}}
\def\vecx{\vec{x}}
\def\vecy{\vec{y}}
\def\vecz{\vec{z}}
\def\vecv{\vec{v}}
\def\vecw{\vec{w}}
\def\vecvh{\widehat{\vecv}}
\def\cX{\mathcal{X}}
\def\agX{\overline{\mathcal{X}}}
\def\agx{\overline{\vecx}}

\newcommand{\x}[2][]{\ifthenelse{\isempty{#1}}{\mathbf{x}_{#2}^k}{\mathbf{x}_{#2}^{k + {#1}}}}
\newcommand{\vv}[2][]{\ifthenelse{\isempty{#1}}{\mathbf{v}_{#2}^k}{\mathbf{v}_{#2}^{k + {#1}}}}
\newcommand{\hatx}[1][i]{\widehat{\mathbf{x}}_{#1}^{k + 1}}
\newcommand{\X}[1][]{\ifthenelse{\isempty{#1}}{X^k}{X^{k + {#1}}}}
\newcommand{\V}[1][]{\ifthenelse{\isempty{#1}}{V^k}{V^{k + {#1}}}}
\newcommand{\barx}[1][]{\ifthenelse{\isempty{#1}}{\overline{\mathbf{x}}^k}{\overline{\mathbf{x}}^{k + {#1}}}}
\newcommand{\hatX}{\widehat{X}^{k + 1}}
\newcommand{\barX}[1][]{\ifthenelse{\isempty{#1}}{\overline{X}^k}{\overline{X}^{k + {#1}}}}
\def\xstar{\mathbf{x}^*}
\def\xbarstar{\overline{\mathbf{x}}^*}

\newcommand{\XX}[1][]{\ifthenelse{\isempty{#1}}{\mathcal{X}_i}{\mathcal{X}_{#1}}}
\newcommand{\F}[2][]{\ifthenelse{\isempty{#1}}{F_{#2}^k}{F_{#2}^{#1}}}
\newcommand{\hatF}[1][]{\ifthenelse{\isempty{#1}}{\widehat{F}^k}{\widehat{F}_{#1}^k}}

\def\Lbar{\overline{L}}
\def\tr{\text{tr}}

\def\T{\mathsf{T}}
\def\bone{{\mathbf{1}}}

\def\fra{\frac{1}{n}}

\def\cG{\mathcal{G}}
\def\cN{\mathcal{N}}
\def\cE{\mathcal{E}}

\newtheorem{lem}{Lemma}
\newtheorem{prop}{Proposition}
\newtheorem{assum}{Assumption}
\theoremstyle{remark}
\newtheorem{rem}{Remark}
\theoremstyle{definition}
\newtheorem{defin}{Definition}

\title{\LARGE \bf
Networked Aggregative Games with Linear Convergence}

\author{Rongping Zhu, Jiaqi Zhang and Keyou You
  \thanks{This work was supported by Technology and Innovation Major Project of the Ministry of Science and Technology of China under Grant 2020AAA0108400 and 2020AAA0108403.}
  \thanks{The authors are with Department of Automation, and BNRist, Tsinghua University, Beijing 100084, China. E-mail: {\tt\small \{zhurp19, zjq16\}@mails.tsinghua.edu.cn}, {\tt\small youky@tsinghua.edu.cn.}}%
}

\begin{document}

\maketitle
\thispagestyle{empty}
\pagestyle{empty}

\begin{abstract}

This paper considers a networked aggregative game (NAG) where the players are distributed over a communication network. By only communicating with a subset of players, the goal of each player in the NAG is to minimize an individual cost function that depends on its own action and the aggregate of all the players' actions.
To this end, we design a novel distributed algorithm that jointly exploits the ideas of the consensus algorithm and the conditional projection descent. 
Under strongly monotone assumption on the pseudo-gradient mapping, the proposed algorithm with fixed step-sizes is proved to  converge linearly to the unique Nash equilibrium of the NAG.
Then the theoretical results are validated by numerical experiments.

\end{abstract}

\section{INTRODUCTION}
\label{sec:intro}

An aggregative game is a non-cooperative Nash game where each player's cost function depends on its own action and the aggregate of all players' actions.
Such games are elaborately introduced in \cite{jensen2006aggregative} and \cite{jensen2010aggregative},
and have been widely applied in resource allocation \cite{alpcan2002game, barrera2015dynamic, basar2016price}, demand response \cite{ma2013decentralized, li2015demand}, etc.
This work considers \textit{Networked Aggregative Games} (NAGs) where players are distributed over a peer-to-peer (P2P) communication network and no central coordinator is included.
To reach the Nash equilibrium (NE) of an NAG,
players only communicates with their neighbors that are defined over the network
to mitigate the lack of global information about the aggregate action.
Note that the NAG is still a non-cooperative game in the sense that each player selfishly minimizes its local cost function.

Distributed seeking of NE has extensively been studied in recent years.
For games where players' cost functions depend on others' actions in a general manner,
researchers have proposed fully distributed algorithms based on alternating direction method of multipliers \cite{salehisadaghiani2019distributed}, 
best-response dynamics \cite{bianchi2020fully}, 
and projected-gradient method \cite{tatarenko2019geometric, tatarenko2020geometric, bianchi2021fully}.
In these algorithms, players utilize local communications to estimate others' actions,
and the NE is found by solving equivalent variational inequality problems \cite{facchinei2007finite}.
Linear convergence to the NE is proved for these algorithms when the games possess strongly monotone pseudo-gradient mappings, which coincides with the results of centralized algorithms \cite{facchinei2007finite}.

However, the above mentioned works require each player to estimate all the other players' actions and exchange its estimates with neighbors,
which is not scalable with the network size and results in large communication overhead for NAGs. Clearly, it is not necessary here as we only need the aggregate action. Moreover, each player in each iteration updates from an inferior point rather than its own action.
The so-called inferior point is obtained from the mixture of each player's own action and neighbor's estimates, which is essential for the consensus part.
Thus, these algorithms may get stuck in the beginning phase with slow convergence rate, which is also confirmed in our numerical experiments in Section \ref{sec:numeric}.

By exploiting the NAG's structure, distributed algorithms with communication-efficiency have been designed.
Researchers in \cite{koshal2016distributed} and \cite{belgioioso2020distributed} adopted a consensus protocol to dynamically track the aggregate action and used projected-gradient methods to update each player's action.
Since only aggregate estimates are exchanged, compared with the algorithms in \cite{salehisadaghiani2019distributed, bianchi2020fully, tatarenko2019geometric, tatarenko2020geometric, bianchi2021fully}, there is an $n$-fold reduction in the amount of data that is transmitted in each iteration, where $n$ is the number of players.
However, the convergence was shown only for diminishing step-sizes with strictly monotone pseudo-gradient mappings.
Authors in \cite{zhu2021asynchronous} proposed an asynchronous consensus-based algorithms and numerically showed the potential of linear convergence for fixed step-sizes.
Nonetheless, theoretical guarantee of linear convergence of NAG still remains an open question \cite{yi2020new}.

In this work, we propose a novel linearly convergent algorithm to find an NE of the NAGs.
The proposed algorithm consists of three key components: (a)
a consensus term to dynamically track the aggregate action; (b)
a projected-gradient step to compute a feasible direction; and (c)
a step along the feasible direction to update the actions.
Then, the linear convergence is proved via establishing a linear system inequality. To the best of our knowledge, our algorithm is the first one that is efficient for NAGs and converges linearly.
Finally, numerical experiments confirm that our algorithm outperforms the existing pseudo-gradient based ones.

The rest of this paper is organized as follows.
In Section \ref{sec:pre}, we introduce the NAGs and some preliminaries.
Section \ref{sec:main} devotes to describing our algorithm with the convergence analysis.
In Section \ref{sec:numeric}, we validate the proposed algorithm on a class of Nash-Cournot games.
Concluding remarks are drawn in Section \ref{sec:concousion}.

\textbf{Notations:} Throughout the paper, we use $\vecx^\T$ and $X^\T$ to denote the transposes of a vector $\vecx$ and a matrix $X$, respectively.
Let $\|\vecx\| = \sqrt{\vecx^\T \vecx}$ denote the Euclidean norm of $\vecx$.
The inner product of matrices $X$ and $Y$ is denoted by $\langle X, Y \rangle_F \triangleq \tr(X^\T Y)$, where $\tr(\cdot)$ denotes the trace of a matrix.
Then, the Frobenius norm of $X$ is defined as $\|X\|_F \triangleq \sqrt{\langle X, X \rangle_F}$.
We use $[\vecx_1, \cdots, \vecx_n]$ and $[\vecx_1; \cdots; \vecx_n]$ to denote the horizontal and vertical stack of vectors $\vecx_1, \ldots, \vecx_n$, respectively.
For matrix $X$, we write $A_{ij}$ to denote its $(i, j)$-th element and $A_{:j}$ to denote its $j$-th column vector.
$\bone$ denotes the vector with all entries equal to 1.
Finally, we use $\Pi_{\cX}[\cdot]$ to denote the Euclidean projection of a vector to a closed convex set $\cX$, i.e., $\Pi_{\cX}[\vecx] \triangleq \arg \min_{\vecz \in \cX} \|\vecx - \vecz\|$.

\section{PRELIMINARIES}
\label{sec:pre}

In this section, we first introduce the NAGs over a multi-agent network.
Then, we reformulate the NAG as a variational inequality which leads to our projection-based algorithm design.
Finally, we introduce some common assumptions.

\subsection{Networked Aggregative Games}

Consider a set of $n$ players indexed by $\mathcal{N} = \{1, 2, \cdots, n\}$.
In an aggregative game, player $i$'s cost function $f_i(\vec{x}_i, \bar{\vec{x}})$ depends on its own action $\vec{x}_i \in \XX \subset \mathbb{R}^p$ and the aggregate action $\bar{\vec{x}} = \frac{1}{n} \sumi \vec{x}_i$.
The objective of player $i$ in the AG is to minimize its local cost:
\begin{equation}
  \label{eq:prob}
  \begin{aligned}
    \text{minimize} & \quad f_i(\vec{x}_i, \bar{\vec{x}}), \\
    \text{subject to} & \quad \vec{x}_i \in \XX,
  \end{aligned}
\end{equation}
where $\vec{x}_j$, $j \neq i$, $j \in \mathcal{N}$ are fixed, and to reach an NE among all players \cite{bacsar1998dynamic}.

\begin{defin}[Nash Equilibrium]
  An $n$-tuple of actions $\vec{x}^* = [\vec{x}_1^*; \vec{x}_2^*; \cdots; \vec{x}_n^*]$ is a Nash equilibrium (NE) of the AG \eqref{eq:prob} if for all $i \in \mathcal{N}$ and $\vec{x}_i \in \XX$,
  $$\textstyle f_i(\vec{x}_i^*, \bar{\vec{x}}^*) \leq f_i(\vec{x}_i, \frac{1}{n}\vec{x}_i + \frac{1}{n} \sumj \vec{x}_j^*).$$
\end{defin}

This work considers the NAG where the players are distributed over a peer-to-peer (P2P) network and have no direct access to the aggregate action.
Instead, every player can only communicate with neighboring players defined by the network.
We model the P2P network by an undirected graph $\cG(\cN, \cE)$ and a weight matrix $W$, where $\cN$ is the set of players and $\cE \subseteq \cN \times \cN$ is the edge set of communications, i.e., $(i, j) \in \cE$ if and only if player $i$ can exchange information directly with player $j$.
The weight matrix $W = \{w_{ij}\}_{n \times n}$ satisfies that $w_{ij} > 0$ if $(i, j) \in \cE$ or $i = j$, and $w_{ij} = 0$ otherwise.
The neighbors of player $i$ is a set $\{j|(i, j) \in \cE\}$.

\subsection{Variational Inequality Approach}
\label{subsec:vi}

The following assumption on the action sets and cost functions is common in the literature \cite{facchinei2007finite}.

\begin{assum}
  \label{as:game}
  For each player $i \in \cN$, the action set $\XX$ is compact and convex.
  Each local cost function $f_i(\vecx_i, \vecz)$ is continuously differentiable in $(\vecx_i, \vecz)$ over some open set containing $\XX \times \overline{\mathcal{X}}$, where
  \begin{equation*}
    \textstyle \agX = \{\fra \sumi \vecx_i|\vecx_i \in \cX_i\}.
  \end{equation*}
  Moreover, the function $\vecx_i \to f_i(\vec{x}_i, \frac{1}{n} \vecx_i + \bar{\vecx}_{-i})$ is convex in $\vecx_i$ over $\XX$ for any fixed $\bar{\vecx}_{-i}$.
\end{assum}

Under Assumption \ref{as:game}, an NE of the AG \eqref{eq:prob} can be reached by solving a variational inequality \cite{facchinei2007finite},
which is to find an $\vecx^* \in \cX$ such that
\begin{equation}
  \label{eq:vi}
  (\vecx - \vecx^*)^\T \phi(\vecx^*) \geq 0, \quad \forall \vecx \in \cX,
\end{equation}
where $\cX = \prod_{i=1}^n \XX$ and $\phi(\vecx)$ is the pseudo-gradient mapping defined as
$$ 
  \phi(\vecx) = [\nabla_{\vecx_1} f_1(\vecx_1, \agx); \nabla_{\vecx_2} f_2(\vecx_2, \agx); \cdots; \nabla_{\vecx_n} f_n(\vecx_n, \agx)].
$$

To emphasize the coordinate of the pseudo-gradient, we define the mapping
$$ F(\vecx, \vecz) = [F_1(\vecx_1, \vecz); F_2(\vecx_2, \vecz); \cdots; F_n(\vecx_n, \vecz)], $$
where
$$ \textstyle F_i(\vecx_i, \vecz) = \nabla_{\vecx_i}f_i(\vecx_i, \vecz) + \fra \nabla_{\vecz}f_i(\vecx_i, \vecz). $$
Clearly, we have $\phi(\vecx) = F(\vecx, \agx)$.

\subsection{Some Assumptions and Lemmas}

We first introduce the two assumptions on the pseudo-gradient mapping, which are necessary for the linear convergence of projection type algorithms even in the centralized setting \cite{facchinei2007finite},
and are commonly used in the literature, see e.g., \cite{salehisadaghiani2019distributed, bianchi2020fully, tatarenko2019geometric, tatarenko2020geometric, bianchi2021fully}.

\begin{assum}[Strong Monotonicity]
  \label{as:monotone}
  The pseudo-gradient $\phi$ is strongly monotone over $\cX$, i.e., there exists a constant $\mu > 0$ such that 
  $$ (\vecx - \vecx')^\T (\phi(\vecx) - \phi(\vecx')) \geq \mu \|\vecx - \vecx'\|^2, \forall \vecx, \vecx' \in \cX. $$
\end{assum}

\begin{assum}[Lipschitz Continuity]
  \label{as:FL}
  Each mapping $F_i$ is uniformly Lipschitz continuous over $\XX \times \agX$, i.e., there exists a constant $L > 0$ such that for any $\vecx_i, \vecx_i' \in \XX$ and $\vecz, \vecz' \in \agX$, it holds
  $$
    \| F_i(\vecx_i, \vecz) - F_i(\vecx_i', \vecz') \| \leq
    L \left\| [\vecx; \vecz] - [\vecx'; \vecz'] \right\|.
  $$
\end{assum}

\begin{lem}[Thm. 2.3.3 in \cite{facchinei2007finite}]
  \label{lm:unique}
  Under Assumptions \ref{as:game} and \ref{as:monotone}, the AG in \eqref{eq:prob} has a unique NE.
\end{lem}

Next, we state an assumption on the communication network.

\begin{assum}[Communication Network]
  \label{as:graph} \hfill
  \begin{enumerate}
    \item The graph $\cG$ is connected, i.e., there exists a sequence of consecutive edges to connect any pair of players.
    \item The weight matrix $W$ is doubly stochastic, i.e., $W \bone = W^\T \bone = \bone$.
  \end{enumerate}
\end{assum}

The first part of Assumption \ref{as:graph} ensures that every player's action can affect every other players eventually, which is necessary in the distributed problems.
The second part induces the averaging property of the weight matrix (see Lemma \ref{lm:spectrum}),
and is common in solving distributed problems, see, e.g., \cite{qu2018harnessing,dong2020distributed}.
The selection of such matrices has been extensively studied in the literature  \cite{olshevsky2009convergence, olfati2007consensus, nedic2009distributed}.

\begin{lem}[Thm. 5.1 in \cite{olshevsky2009convergence}]
  \label{lm:spectrum}
  Under Assumption \ref{as:graph}, there exists a constant $\sigma \in (0, 1)$ which is the spectral norm of $W - \fra \bone \bone^\T$, or the second largest singular value of $W$, such that for any $\vecw \in \mathbb{R}^n$,
  $$ \textstyle \|W\vecw - \bone \bar{w}\| = \|(W - \fra \bone \bone^\T)(\vecw - \bone \bar{w})\| \leq \sigma \|\vecw - \bone \bar{w}\|, $$
  where $\bar{w} = \fra \bone^\T \vecw$.
\end{lem}

We also use the following property of projection operators.

\begin{lem}[Prop. 1.1.4 in \cite{bertsekas2016nonlinear}]
  \label{lm:projection}
  Let $\cX \subset \mathbb{R}^p$ be a nonempty closed convex set. Then for any $\vecx, \vecy \in \mathbb{R}^p$, we have that $\|\Pi_{\cX}[\vecx] - \Pi_{\cX}[\vecy]\| \leq \|\vecx - \vecy\|$.
\end{lem}

\section{MAIN RESULTS}
\label{sec:main}

In this section, we propose a novel distributed algorithm for the NAG in \eqref{eq:prob} and prove its linear convergence.

\subsection{Algorithm Design}

It follows from Section \ref{subsec:vi} that the NAG corresponds to the variational inequality in \eqref{eq:vi}.
We recall a standard centralized algorithm obtained by a fixed-point iteration:
\begin{equation}
  \label{eq:central}
  \x[1]{i} = \Pi_{\XX}[\x{i} - \alpha F_i(\x{i}, \barx)], \quad \forall i \in \cN.
\end{equation}
If the step-size $\alpha$ is chosen small enough, the sequences $\{\vecx^k\}_{k \in \mathbb{N}}$ converges to a fixed point of the mapping $\vecx \mapsto \Pi_{\cX}[\vecx - \alpha \phi(\vecx)]$, which is exactly $\vecx^*$ \cite{facchinei2007finite}.

However, as mentioned in the previous section, the players in the NAG have no direct access to the aggregate action $\agx$.
Thus, for all $i \in \cN$, let player $i$ maintain a local vector $\vecv_i$ and exchange it with its neighbors to estimate the aggregate action.
Then, they use the estimated aggregate action to locally update their actions.
The details are summarized in Algorithm \ref{algo}, where $\alpha > 0$ and $\beta \in (0, 1]$ are two fixed step-sizes.


\begin{algorithm}[t!]
  \caption{The distributed algorithm for the NAG \label{algo}}
  \begin{itemize}[leftmargin=*,label={}]
    \item \textbf{Initialization}: for all $i \in \cN$, set $\vecx_i^0 \in \XX$ and $\vecv_i^0 = \vecx_i^0$.
    \item \textbf{Repeat until convergence}: for all $i \in \cN$,
      \begin{subequations}
        \label{eq:algo}
        \begin{align}
            \vecvh_i^{k + 1} &=  \textstyle \sum_{j = 1}^n w_{ij} \vv{j}, \label{eq:algo-w} \\
            \hatx &= \Pi_{\XX}[\x{i} - \alpha F_i(\x{i}, \vecvh_i^{k + 1})], \label{eq:algo-hatx} \\
            \x[1]{i} &= \x{i} + \beta (\hatx - \x{i}), \label{eq:algo-x} \\
            \vv[1]{i} &= \vecvh_i^{k + 1} + \x[1]{i} - \x{i} \label{eq:algo-v}
        .\end{align}
      \end{subequations}
  \end{itemize}
\end{algorithm}

The main idea of Algorithm \ref{algo} is very natural.
Similar to other consensus-based algorithms, e.g., \cite{nedic2009distributed, qu2018harnessing,dong2020distributed},
the distributed averaging in \eqref{eq:algo-w} and \eqref{eq:algo-v} drives both $\vecvh_i$ and $\vecv_i$ to the aggregate action $\agx$,
while $\vecvh_i$ is a more accurate estimate due to one more mixing step in \eqref{eq:algo-w}.
Then, a projected-gradient step is conducted in \eqref{eq:algo-hatx} to compute a feasible direction $\hatx - \x{i}$,
and the action is updated along the direction with a step-size $\beta$ in (\ref{eq:algo-x}).
Note that $\x[1]{i}$ is a convex combination of $\x{i}$ and $\hatx$, thus also belongs to $\XX$.

Our key idea for designing Algorithm \ref{algo} lies in (\ref{eq:algo-x}) where $\beta\in(0,1)$.
Note that if $\beta = 1$ and $\alpha$ is replaced by a sequence of diminishing step-sizes;
then Algorithm \ref{algo} is reduced to the algorithm in \cite{koshal2016distributed}.
When $\beta < 1$, each player updates its action by a smaller step along the feasible direction.
As shown in Proposition \ref{prop:rate},
when the second largest singular value $\sigma$ of the consensus matrix $W$ is too large,
we necessarily need $\beta < 1$ to obtain a linear convergence.
Moreover, we can further adjust $\beta$ to reach a faster convergence rate in Section \ref{sec:numeric}.



\begin{rem}
  Consensus-based projected-gradient method is also exploited for solving distributed constrained optimization problems where multiple computational nodes cooperate to optimize a common cost function, see \cite{nedic2010constrained, liu2017convergence, dong2020distributed} for details.
  However, the algorithm cannot be directly applied to NAG since each player needs to optimize an individual cost function.
\end{rem}

\subsection{Convergence Analysis}

Firstly, we provide two lemmas that are important for the later analysis.
Then, we prove the convergence of Algorithm \ref{algo} in \eqref{eq:algo} via constructing a linear system inequality.
Finally, we derive the feasible region of the step-sizes $\alpha$ and $\beta$, and explicitly show the convergence rate.

\begin{lem}
  \label{lm:average}
  $\sumi \vv{i} = \sumi \x{i}$ holds for all $k$.
\end{lem}
\begin{proof}
  Taking the summation of \eqref{eq:algo-v} over $i$, we obtain that
  $$ \begin{aligned}
    \textstyle \sumi \vv[1]{i} &= \textstyle \sum_{j=1}^n (\sumi w_{ij}) \vv{j} \\
    & \quad + \textstyle \sumi \x[1]{i} - \sumi \x{i} \\
    &= \textstyle \sumi \x[1]{i} + \sumi \vv{i} - \sumi \x{i}.
  \end{aligned} $$
  Since each player assigns $\vecv_i^0 = \vecx_i^0$, the result is arrived at by induction.
\end{proof}

\begin{lem}
  \label{lm:ag-Lip}
  Under Assumption \ref{as:FL}, the pseudo-gradient $\phi$ is Lipschitz continuous over $\cX$ with coefficient $\Lbar = \sqrt{2}L$.
\end{lem}
\begin{proof}
  From the definition of the pseudo-gradient in Section \ref{subsec:vi} we have that for any $\vecx, \vecx' \in \cX$,
  \begin{equation*}
    \begin{aligned}
      \|\phi(\vecx) - \phi(\vecx')\|^2 &= \textstyle \sumi \|F_i(\vecx_i, \agx) - F_i(\vecx_i', \agx')\|^2 \\
      & \leq L^2 \textstyle \sumi (\|\vecx_i - \vecx_i'\|^2 + \|\agx - \agx'\|^2) \\
      & \leq 2 L^2 \textstyle \|\vecx - \vecx'\|^2.
    \end{aligned}
  \end{equation*}
  Thus, $\phi$ is Lipschitz continuous with coefficient $\sqrt{2}L$.
\end{proof}

In the following analysis, we denote
\begin{equation*}
  \begin{aligned}
    \X &= [\x{1}, \x{2}, \cdots, \x{n}]^\T, &
      \widehat{X}^k &= [\widehat{\vecx}_1^k, \widehat{\vecx}_2^k, \cdots, \widehat{\vecx}_n^k]^\T, \\
    X^* &= [\xstar_1, \xstar_2, \cdots, \xstar_n]^\T, & \V &= [\vv{1}, \vv{2}, \cdots, \vv{n}]^\T,
  \end{aligned}
\end{equation*}
and $\barX = \bone (\barx)^\T$.
Then,  \eqref{eq:algo-x} and \eqref{eq:algo-v} can be compactly written as
\begin{equation*}
  \begin{aligned}
    \X[1] &= \beta \hatX + (1 - \beta) \X, \\
    \V[1] &= W \V + \X[1] - \X.
  \end{aligned}
\end{equation*}

\begin{prop}
  \label{prop:linearSys}
  Under Assumptions \ref{as:game}-\ref{as:graph}, when $\alpha$ and $\beta$ are small enough such that the matrix
  \begin{equation*}
    M = \begin{bmatrix}
      1 - \alpha \beta (\mu - 2 \alpha \Lbar^2) &
        \alpha \beta L^2 (1 / \mu + 2 \alpha) \\
      \frac{8\beta^2 - 4 \alpha \beta^3 (\mu - 2 \alpha \Lbar^2)}{1 - \sigma^2} &
        \frac{4\alpha \beta^3 L^2 (1 / \mu + 2 \alpha)}{1 - \sigma^2} + \frac{2 \sigma^2}{1 + \sigma^2}
    \end{bmatrix}
  \end{equation*}
  has spectral radius $\rho(M) < 1$.
  Then, $\|\X - X^*\|_F^2$ and $\|\V - \barX\|_F^2$ are all decaying with rate $O(\rho(M)^k)$.
\end{prop}

\begin{proof}
We first bound $\|\X[1] - X^*\|_F^2$ and $\|\V[1] - \barX[1]\|_F^2$ by linear combinations of their values in the last iteration.
Then, a linear system inequality is derived which implies the convergence rate. 

\textbf{Step 1: Bound $\|\X[1] - X^*\|_F^2$.} It follows from \eqref{eq:algo-hatx} that
\begin{equation}
  \label{eq:hatxBound}
  \begin{aligned}
    & \|\hatx - \xstar_i\|^2 \\
    & = \|\Pi_{\XX}[\x{i} - \alpha \hatF[i]] - \Pi_{\XX}[\xstar_i - \alpha \F[*]{i}]\|^2 \\
    & \leq \|\x{i} - \xstar_i - \alpha(\hatF[i] - \F[*]{i})\|^2 \\
    & = \|\x{i} - \xstar_i\|^2 + \alpha^2 \|\hatF[i] - \F[*]{i}\|^2 \\
    & \quad - 2 \alpha (\x{i} - \xstar_i)^\T(\hatF[i] - \F[*]{i}),
  \end{aligned}
\end{equation}
where we have used $\hatF[i]$ and $\F[*]{i}$ to respectively denote $F_i(\x{i}, \vecvh_i^{k + 1})$ and $F_i(\xstar_i, \xbarstar)$ for simplicity.
The inequality follows from Lemma \ref{lm:projection}.
Denote $F_i(\x{i}, \barx)$ by $\F{i}$; then
\begin{equation}
  \label{eq:FBreak}
  \begin{aligned}
    \|\hatF[i] - \F[*]{i}\|^2 & = \|\hatF[i] - \F{i}\|^2 + \|\F{i} - \F[*]{i}\|^2 \\
    & \quad + 2(\hatF[i] - \F{i})^\T (\F{i} - \F[*]{i}) \\
    & \leq L^2 \|\vecvh_i^{k + 1} - \barx\|^2 + \|\F{i} - \F[*]{i}\|^2 \\
    & \quad + 2(\hatF[i] - \F{i})^\T (\F{i} - \F[*]{i})
  \end{aligned}
\end{equation}
and
\begin{equation}
  \label{eq:crossBreak}
  \begin{aligned}
    & (\x{i} - \xstar_i)^\T(\hatF[i] - \F[*]{i}) \\
    & = (\x{i} - \xstar_i)^\T(\hatF[i] - \F{i}) + (\x{i} - \xstar_i)^\T(\F{i} - \F[*]{i}) .
  \end{aligned}
\end{equation}
Plugging \eqref{eq:FBreak}, \eqref{eq:crossBreak} into \eqref{eq:hatxBound} and taking the summation over $i$, we have
\begin{equation*}
  \begin{aligned}
    &\|\hatX - X^*\|_F^2 \\
    & \leq \|\X - X^*\|_F^2 + \alpha^2 L^2 \|W\V - \barX\|_F^2 \\
    & \quad + \alpha^2 \|F^k - F^*\|_F^2 - 2 \alpha \cdot \langle \X - X^*, F^k - F^* \rangle_F \\
    & \quad + 2 \alpha^2 \cdot \langle \hatF - F^k, F^k - F^* \rangle_F \\
    & \quad - 2 \alpha \cdot \langle \X - X^*, \hatF - F^k \rangle_F \\
    & \leq \|\X - X^*\|_F^2 + \alpha^2 L^2 \|W\V - \barX\|_F^2 \\
    & \quad + \alpha^2 \Lbar^2 \|\X - X^*\|_F^2 - 2 \alpha \mu \|\X - X^*\|_F^2 \\
    & \quad + \alpha^2 \Lbar^2 \|\X - X^*\|_F^2 + \alpha^2 L^2 \|W\V - \barX\|_F^2 \\
    & \quad + \alpha \mu \|\X - X^*\|_F^2 + \alpha (L^2 / \mu) \cdot \|W\V - \barX\|_F^2,
  \end{aligned}
\end{equation*}
where the second inequality dues to Cauchy-Schwarz inequality and mean inequality.
Namely,
\begin{equation}
  \label{eq:hatXBound}
  \begin{aligned}
    \|\hatX - X^*\|_F^2 &\leq (1 - \alpha \mu + 2 \alpha^2 \Lbar^2) \|\X - X^*\|_F^2 \\
    & \quad + (\alpha / \mu + 2 \alpha^2) L^2 \sigma^2 \|\V - \barX\|_F^2,
  \end{aligned}
\end{equation}

It follows from \eqref{eq:algo-x} that
\begin{equation}
  \label{eq:XwithhatX}
  \begin{aligned}
    & \|\X[1] - X^*\|_F^2 \\
    & = \beta^2 \|\hatX - X^*\|_F^2 + (1 - \beta)^2 \|\X - X^*\|_F^2 \\
    & \quad  + 2 \beta (1 - \beta) (\hatX - X^*)^\T(\X - X^*) \\
    & \leq \beta \|\hatX - X^*\|_F^2 + (1 - \beta) \|\X - X^*\|_F^2.
  \end{aligned}
\end{equation}
Plugging \eqref{eq:hatXBound} into \eqref{eq:XwithhatX}, we have
\begin{equation}
  \label{eq:XBound}
  \begin{aligned}
    \|\X[1] - X^*\|_F^2 &\leq (1 - a) \|\X - X^*\|_F^2 \\
    & \quad + b \|\V - \barX\|_F^2,
  \end{aligned}
\end{equation}
where
\begin{equation}
  \label{eq:def-a-b}
  \begin{aligned}
    a &= \alpha \beta (\mu - 2 \alpha \Lbar^2) \\
    b &= \alpha \beta \sigma^2 L^2 (1 / \mu + 2 \alpha).
  \end{aligned}
\end{equation}

\textbf{Step 2: Bound $\|\V[1] - \barX[1]\|_F^2$.} By the update rule \eqref{eq:algo-v},
\begin{equation}
  \label{eq:Vnorm}
  \begin{aligned}
    & \|\V[1]_{:j} - \barX[1]_{:j}\| \\
    & \leq \textstyle \|(W - \fra \bone \bone^\T)(\V_{:j} - \barX_{:j})\| \\
    & \quad + \|\X[1]_{:j} - \X_{:j} - \barX[1]_{:j} + \barX_{:j}\| \\
    & \leq \sigma \|\V_{:j} - \barX_{:j}\| \\
    & \quad + \beta \|\hatX_{:j} - \X_{:j} - \textstyle \frac{1}{n} \bone_n \cdot \sumi(\hatX_{ij} - \X_{ij})\|.
  \end{aligned}
\end{equation}
It is easy to verify that
\begin{equation*}
  \begin{aligned}
    & \|\hatX_{:j} - \X_{:j} - \textstyle \frac{1}{n} \bone_n \cdot \sumi(\hatX_{ij} - \X_{ij})\|^2 \\
    & = \|\hatX_{:j} - \X_{:j}\|^2 + \textstyle \frac{1}{n} [\sumi(\hatX_{ij} - \X_{ij})]^2 \\
    & \quad - \textstyle \frac{2}{n} [\sumi(\hatX_{ij} - \X_{ij})] \cdot \bone_n^\T(\hatX_{:j} - \X_{:j}) \\
    & \leq \|\hatX_{:j} - \X_{:j}\|^2.
  \end{aligned}
\end{equation*}
Plugging this into \eqref{eq:Vnorm} and squaring the both sides, we have
\begin{equation}
  \label{eq:Vj}
  \begin{aligned}
    & \|\V[1]_{:j} - \barX[1]_{:j}\|^2 \\
    & \leq \textstyle \frac{2 \sigma^2}{1 + \sigma^2} \|\V_{:j} - \barX_{:j}\|^2 + \frac{2\beta^2}{1 - \sigma^2} \|\hatX_{:j} - \X_{:j}\|^2 \\
    & \leq \textstyle \frac{2 \sigma^2}{1 + \sigma^2} \|\V_{:j} - \barX_{:j}\|^2 + \frac{4\beta^2}{1 - \sigma^2} \|\X_{:j} - X^*_{:j}\|^2 \\
    & \quad + \textstyle \frac{4\beta^2}{1 - \sigma^2} \|\hatX_{:j} - X^*_{:j}\|^2.
  \end{aligned}
\end{equation}
Letting
\begin{equation}
  \label{eq:def-c-d}
    c = \frac{4\beta^2}{1 - \sigma^2}, \quad d = \frac{2 \sigma^2}{1 + \sigma^2},
\end{equation}
taking the summation over $j$ in \eqref{eq:Vj}, and combining \eqref{eq:hatXBound}, we have
\begin{equation}
  \label{eq:VBound}
  \begin{aligned}
    & \|\V[1] - \barX[1]\|_F^2 \\
    & \leq d \|\V - \barX\|_F^2 + c \|\X - X^*\|_F^2 \\
    & \quad + c \|\hatX - X^*\|_F^2 \\
    & \leq c(2-a) \|\X - X^*\|_F^2 + (d + cb) \|\V - \barX\|_F^2.
  \end{aligned}
\end{equation}

\textbf{Step 3: Derive a linear system inequality.} Combining \eqref{eq:XBound} and \eqref{eq:VBound}, we have
\begin{equation}
  \label{eq:linearSys}
  \begin{aligned}
    & \begin{bmatrix}
      \|\X[1] - X^*\|_F^2 \\ \|\V[1] - \barX[1]\|_F^2
    \end{bmatrix} \\
    & \leq \overbrace{\begin{bmatrix}
      1 - a & b \\
      c(2 - a) & bc + d
    \end{bmatrix}}^{\triangleq M \in \mathbb{R}^{2 \times 2}}
    \overbrace{\begin{bmatrix}
      \|\X - X^*\|_F^2 \\ \|\V - \barX\|_F^2
    \end{bmatrix}}^{\triangleq \zeta^k \in \mathbb{R}^{2}},
  \end{aligned}
\end{equation}
where $a$, $b$, $c$ and $d$ are defined in \eqref{eq:def-a-b} and \eqref{eq:def-c-d}.
Since $M$ and $\zeta^k$ have non-negative entries, we can expand \eqref{eq:linearSys} recursively and obtain
\begin{equation*}
  \zeta^k \leq M^k \zeta^0.
\end{equation*}

Thus, if $\alpha$ and $\beta$ is such that $\rho(M) < 1$; then
$\|\X - X^*\|_F^2$ and $\|\V - \barX\|_F^2$ are all decaying with rate $O(\rho(M)^k)$.
\end{proof}

Next we explicitly show the convergence rate of Algorithm \ref{algo}.

\begin{prop}
  \label{prop:rate}
  When
  \begin{equation}
    \label{eq:ss-bound}
    \begin{aligned}
      0 &< \beta \leq \min \left\{ 1, \frac{\mu (1 - \sigma^2)}{2 \sigma L \sqrt{7 + 11 \sigma^2}} \right\}, \\
      0 &< \alpha \leq \frac{\mu^2(1 - \sigma^2)^2 - 4 (7 + 11 \sigma^2) (\sigma L \beta)^2}{2 \mu \Lbar^2 (1 - \sigma^2)^2 + 8 \mu (7 + 11 \sigma^2) (\sigma L \beta)^2},
    \end{aligned}
  \end{equation}
  we have
  $$ \rho(M) < 1 - \frac{\alpha \beta}{2} \left[ \mu - 2 \alpha \Lbar^2 - \frac{4 (\sigma L \beta)^2}{1 - \sigma^2} \left( \frac{1}{\mu} + 2 \alpha \right) \right] < 1. $$
\end{prop}

\begin{proof}
Since $0 < \alpha < \frac{\mu}{2 \Lbar^2}$, $0 < a < 1$, the characteristic polynomial $p(\lambda)$ of $M$ is
$$ p(\lambda) = \lambda^2 + (a - bc - d - 1) \lambda + [(1 - a)d - bc], $$
where $a$, $b$, $c$ and $d$ are defined in \eqref{eq:def-a-b} and \eqref{eq:def-c-d}.
The two roots of $p(\lambda)$ are $\frac{-B \pm \sqrt{\Delta}}{2}$, where
$$ \begin{aligned}
  B & = a - bc - d - 1 < 0, \\
  \Delta & = B^2 - 4[(1 - a)d - bc] > 0.
\end{aligned} $$
Thus, the spectral radius of $M$ is $\rho(M) = \frac{-B + \sqrt{\Delta}}{2}$.
Since
$$\beta \leq 1 < \frac{8 \Lbar^2 (1 + \sigma^2)}{\mu^2 (1 - \sigma^2)};$$
we have $a < 1 - d$.
Besides, in view of the upper bound of $\alpha$ and $\beta$ in \eqref{eq:ss-bound}, we have $a(1 - d) > (7 + 2d) bc$ and $a > bc$.
Then, we can obtain that $\Delta < (1 - d) ^2$. Thus,
\begin{equation*}
  \rho(M) < 1 - \frac{a - bc}{2} < 1,
\end{equation*}
which is the desired result.
\end{proof}

\begin{rem}
  When $\sigma$ of the consensus matrix $W$ is smaller, we can obtain a faster convergence rate.
  Besides, if $\sigma$ is small enough such that we can choose $\beta = 1$, then Algorithm \ref{algo} can be simplified.
  Otherwise, we must set $\beta < 1$ to theoretically guarantee the linear convergence.
\end{rem}

\section{NUMERICAL EXPERIMENTS}
\label{sec:numeric}

We consider the Nash-Cournot games in \cite{koshal2016distributed}.
Consider $n = 20$ firms competing over $L = 10$ locations, each firm $i$ needs to decide its production $g_{il}$ and sales $s_{il}$ at location $l$.
The cost of production of firm $i$ at location $l$ is $c_{il}(g_{il})$ and is defined as
$$ c_{il}(g_{il}) = a_{il} g_{il} + b_{il} g_{il}^2, $$
where $a_{il}$ and $b_{il}$ are parameters for firm $i$.
Firm $i$'s revenue at location $l$ is $s_{il}p_l(s_l)$, where $s_l = \sumi s_{il}$ is the total sales at location $l$ and $p_l$ is the corresponding price function.
The price function $p_l$ captures the reverse demand function and is defined as
$$ p_l(s_l) = d_l - s_l, $$
where $d_l$ is a parameter for location $l$.
The production capacity of firm $i$ at location $l$ is denoted by $\tcap_{il}$.
Assuming that the transportation fees between any two locations are zero; then, firm $i$'s optimization problem is:
\begin{equation*}
  \begin{aligned}
    \text{minimize} & \quad \textstyle \sum_{l=1}^L [c_{il}(g_{il}) - s_{il} \cdot p_l(s_l)], \\
    \text{subject to} & \quad g_{il},s_{il}\geq 0,\ g_{il}\leq \text{cap}_{il}, \\
    & \quad \textstyle \sum_{l=1}^L g_{il} = \sum_{l=1}^L s_{il}.
  \end{aligned}
\end{equation*}
It is easy to verify that the problem and the corresponding pseudo-gradient mapping satisfy Assumptions \ref{as:game}-\ref{as:FL}.
In the simulations, we set $a_{il} \sim U(2, 12)$, $b_{il} \sim U(2, 3)$, $d_l \sim U(90, 100)$ and $\tcap_{il} = 500$ for $i = 1, \cdots n$ and $l = 1, \cdots, L$, where $U(u_1, u_2)$ is the uniform distribution over the interval $[u_1, u_2]$.

We first validate Algorithm \ref{algo} over three network topologies:
\begin{itemize}
  \item \textit{Linear}: player $i$ communicates with players $\text{mod}(i + 11j, n) + 1$, $0 \leq j < n / 11$.
  \item \textit{Log}: player $i$ communicates with players $\text{mod}(i + 2^j, n) + 1$, $0 \leq j < \log_2 n$.
  \item \textit{Complete}: every player communicates with every other players.
\end{itemize}
The weight matrix $W = \{w_{ij}\}$ is set as
$$ w_{ij} = \begin{cases}
  0, & \text{if } (i, j) \notin \cE, \\
  \delta, & \text{if } (i, j) \in \cE \text{ and } i \neq j, \\
  1 - \delta \cdot d(i), & \text{if } i = j,
\end{cases} $$
where $d(i)$ is the number of players communicating with player $i$, and
$ \delta = 0.5 / \max_{i} \{d(i)\}. $

Fig. \ref{fig:topos} depicts the trajectories of suboptimality gap $\|\vecx^k - \vecx^*\|$ versus iterations over the three network topologies.
All the trajectories start with the same initial point, and the step-sizes are tuned to be optimal.
It can be seen that Algorithm \ref{algo} converges linearly over different topologies.
Moreover, when the network has stronger connectivity, i.e., a smaller $\sigma$ that defined in Lemma \ref{lm:spectrum}, Algorithm \ref{algo} achieves a faster convergence rate.

\begin{figure}[t!]
  \centering
  \includegraphics[width=.65\linewidth]{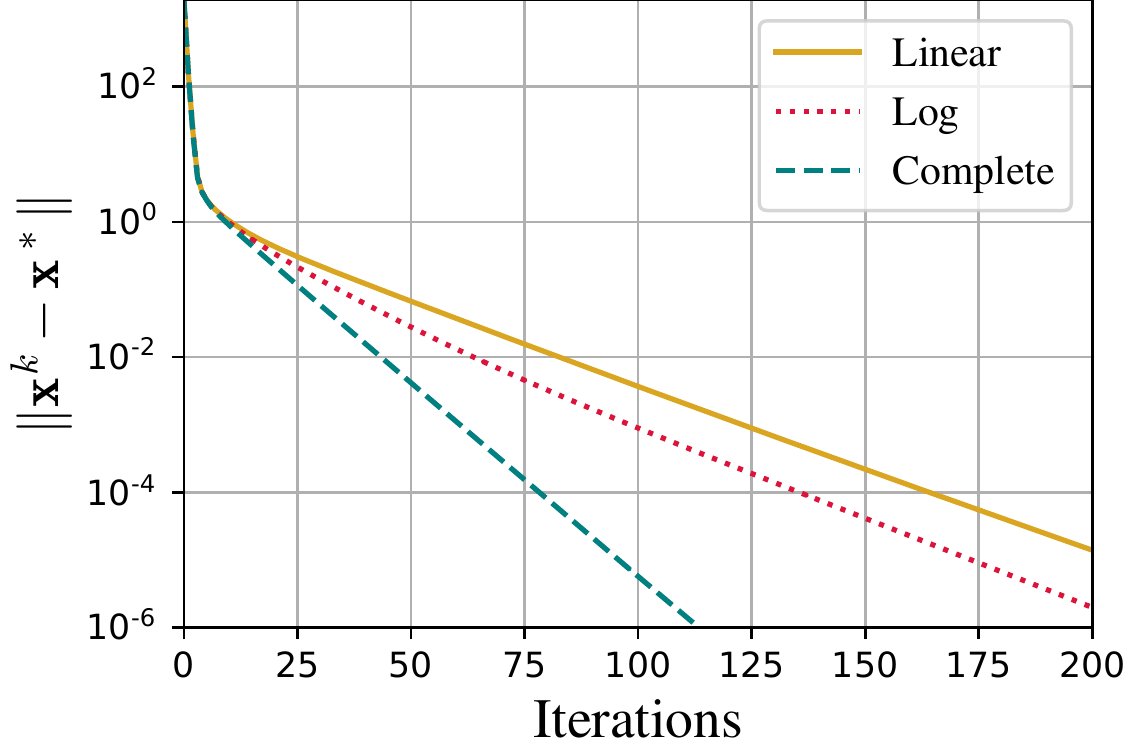}
  \caption{The trajectories of suboptimality gap vs. iterations over different network topologies.}
  \label{fig:topos}
\end{figure}

Then, over the \textit{linear} topology, we fix different values of $\beta$ and search the corresponding optimal value of $\alpha$.
Then, we run Algorithm \ref{algo} for 200 iterations and obtain the suboptimality gap $\epsilon_{200} = \|\vecx^{200} - \vecx^*\|$.
As shown in Table \ref{tb:beta}, a smaller $\beta$ usually leads to a bigger $\alpha$, which coincides with Proposition \ref{prop:rate}.
It cannot be theoretically proved whether Algorithm \ref{algo} converges linearly for $\beta = 1$.
However, our experiment shows that when $\beta$ is slightly less than 1, we can reach a faster convergence rate, i.e., obtain a smaller $\epsilon$.

\begin{table}[t!]
  \caption{Suboptimality Gap on Termination vs. Step-sizes.}
  \label{tb:beta}
  \centering
  \begin{tabular}{|c||c|c|c|c|c|}
    \hline
    $\beta$ & 0.8 & 0.85 & 0.9 & 0.95 & 1.0 \\
    \hline
    $\alpha$ & 0.108 & 0.092 & 0.096 & 0.090 & 0.087 \\
    \hline
    $\epsilon_{200}$ ($\times 10^{-5}$) & 1.615 & 1.622 & \textbf{1.418} & 1.512 & 1.591 \\
    \hline
  \end{tabular}
\end{table}

Finally, still over the \textit{linear} topology, we compare Algorithm \ref{algo} with algorithms in \cite{bianchi2021fully} and \cite{koshal2016distributed},
of which the former is designed for games where cost functions depend on actions in a general manner,
and the latter is designed for NAGs but only with diminishing step-sizes.
Fig. \ref{fig:compare} shows that Algorithm \ref{algo} outperforms the others.
Algorithm in \cite{bianchi2021fully} gets stuck in the beginning phase, which is due to each players' mixture of neighbors faulty estimates of their own action.
Although it converges linearly in the long run, the convergence rate is slower than Algorithm \ref{algo}.
As for algorithm in \cite{koshal2016distributed}, since it uses diminishing step-sizes,
the convergence rate is sublinear and also inferior to Algorithm \ref{algo}.

\begin{figure}[t!]
  \centering
  \includegraphics[width=.65\linewidth]{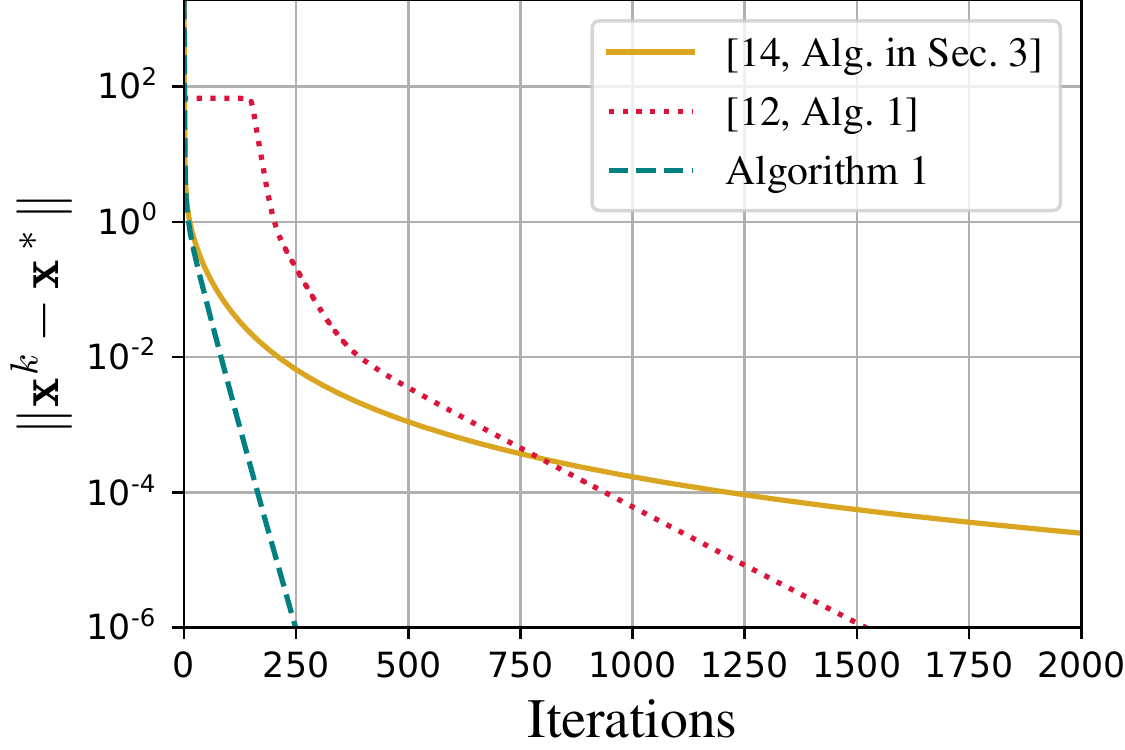}
  \caption{The trajectories of suboptimality gap vs. iterations over different network topologies.}
  \label{fig:compare}
\end{figure}

\section{CONCLUSIONS AND FUTURE WORKS}
\label{sec:concousion}

In this paper, we have proposed a fully distributed algorithm for the NAGs over an undirected communication network.
When the NAG possesses a strongly monotone and Lipschitz continuous pseudo-gradient mapping, the proposed algorithm with fixed step-sizes is proved to converge linearly to the unique NE.
Theoretical results have been validated by numerical experiments.

Future work will consider linearly convergent algorithms for NAGs over directed communication networks or time-varying networks.
Fully asynchronous extension of the proposed algorithm is also an open and difficult problem.




\bibliographystyle{ieeetr}
\bibliography{ref}

\end{document}